\documentclass[12pt]{amsart}

\usepackage{amsmath,amsthm,amsfonts,amscd,flafter,epsf}

\usepackage{graphicx} 

\usepackage{amssymb} 

\usepackage{verbatim}  

\usepackage{colonequals}  

\usepackage[all]{xy}

\usepackage{pdfsync}

\usepackage{hyperref}
 
\usepackage{color}
 
\title[{Hyperbolic Dimension and Decomposition Complexity}]
{Hyperbolic Dimension and Decomposition Complexity} 

 \author[Nicas]{Andrew Nicas$^*$}
\address{Department of Mathematics and Statistics, McMaster University, Hamilton, Ontario, Canada L8S 4K1}
\email{nicas@mcmaster.ca}
\thanks{$^*$Partially supported by a grant from
the Natural Sciences and Engineering Research Council of Canada}
 
\author[Rosenthal]{David Rosenthal$^\dagger$}
\address{Department of Mathematics and Computer Science, St.\ John's University, 8000 Utopia
Pkwy, Queens, NY 11439, USA}
\thanks{$^\dagger$Partially supported by a grant from the Simons Foundation, \#229577.}
\email{rosenthd@stjohns.edu}


\numberwithin{equation}{section}
    
\begin{document}

\begin{abstract} 
The aim of this paper is to provide some new tools to aid the study of {\it decomposition complexity}, a notion introduced by Guentner, Tessera and Yu. In this paper, three equivalent definitions for decomposition complexity are established.
We prove that metric spaces with finite 
hyperbolic dimension have finite (weak) decomposition complexity, and we prove that the collection of metric families that are coarsely embeddable into Hilbert space is closed under decomposition.
A method for showing that certain metric spaces do not have finite decomposition complexity is also discussed.
\end{abstract}

\dedicatory{Dedicated to Ross Geoghegan on the occasion of his $70^{th}$ birthday}

\maketitle

\baselineskip 18pt


\newtheorem{theorem}{Theorem}[section]
\newtheorem*{thm}{Theorem}
\newtheorem{lemma}[theorem]{Lemma}
\newtheorem{proposition}[theorem]{Proposition}
\newtheorem{corollary}[theorem]{Corollary}
\newtheorem{claim}[theorem]{Claim}
\newtheorem{question}[theorem]{Question}

\theoremstyle{definition}
\newtheorem{definition}[theorem]{Definition}

\theoremstyle{definition}
\newtheorem{remark}[theorem]{Remark}

\theoremstyle{remark}
\newtheorem{notation}[theorem]{Notation}

\theoremstyle{definition}
\newtheorem{example}[theorem]{Example}

\newtheorem*{conditionA}{Condition (A)}
\newtheorem*{conditionB}{Condition (B)}
\newtheorem*{conditionC}{Condition (C)}

\newcommand{\ra}{{\rightarrow}}
\newcommand{\bs}{{\backslash}}
\newcommand{\lra}{{\longrightarrow}}


\newcommand{\asdim}{\operatorname{asdim}}
\newcommand{\hyperdim}{\operatorname{hyperdim}}
\newcommand{\whyperdim}{\operatorname{w-hyperdim}}
\newcommand{\id}{\operatorname{id}}
\newcommand{\supp}{\operatorname{supp}}
\newcommand{\st}{\operatorname{star}}
\newcommand{\ev}{\operatorname{ev}}
\newcommand{\mesh}{\operatorname{mesh}}
\newcommand{\diam}{\operatorname{diam}}

\newcommand{\tb}{{\large \textbullet}}

\newcommand{\Zhalf}{{\mathbb Z}[1/2]}


\newcommand{\cB}{\mathcal{B}}
\newcommand{\cC}{\mathcal{C}}
\newcommand{\cD}{\mathcal{D}}
\newcommand{\cE}{{\mathcal E}}
\newcommand{\cF}{{\mathcal F}}
\newcommand{\cG}{{\mathcal G}}
\newcommand{\cH}{{\mathcal H}}
\newcommand{\cJ}{{\mathcal J}}
\newcommand{\cL}{{\mathcal L}}
\newcommand{\calo}{\mathcal{O}}
\newcommand{\cP}{{\mathcal P}}
\newcommand{\cT}{{\mathcal T}}
\newcommand{\cU}{{\mathcal U}}
\newcommand{\cV}{{\mathcal V}}
\newcommand{\cW}{{\mathcal W}}
\newcommand{\cX}{{\mathcal X}}
\newcommand{\cY}{{\mathcal Y}}
\newcommand{\cZ}{{\mathcal Z}}

\newcommand{\fA}{\mathfrak{A}}
\newcommand{\fB}{\mathfrak{B}}
\newcommand{\fC}{\mathfrak{C}}
\newcommand{\fD}{\mathfrak{D}}
\newcommand{\wD}{w\mathfrak{D}}
\newcommand{\fE}{\mathfrak{E}}
\newcommand{\fH}{\mathfrak{H}}
\newcommand{\fL}{\mathfrak{L}}
\newcommand{\N}{\mathfrak{N}}
\newcommand{\fU}{\mathfrak{U}}


\newcommand{\tS}{{\widetilde S}}


\newcommand{\B}{{\rm B}}
\newcommand{\E}{{\rm E}}

\newcommand{\K}{{\mathbf K}}
\newcommand{\HH}{{\mathbf{HH}}}

\newcommand{\bbH}{{\mathbb H}}
\newcommand{\bN}{{\mathbb N}}
\newcommand{\R}{{\mathbb R}}
\newcommand{\Z}{{\mathbb Z}}




\section{Introduction}

The {\it asymptotic dimension} of a metric space was introduced by Gromov  \cite{Gromov}
as a tool for studying the large scale geometry of groups.
Interest in this concept intensified when Guoliang Yu proved the Novikov Conjecture
for a finitely generated group $G$ having finite asymptotic
dimension as a metric space with a word-length metric and whose classifying space $BG$ has the homotopy type
of a finite complex,~\cite{Yu}. There are many geometrically interesting metric spaces that do not have finite asymptotic dimension.
In order to study groups with infinite asymptotic dimension, 
Guentner, Tessera and Yu 
introduced the notion of {\it finite decomposition complexity}, abbreviated here to {\it FDC},~\cite{Guentner_Tessera_Yu1}.
Every countable group admits a proper left-invariant metric that is unique up to coarse equivalence.
Guentner, Tessera and Yu  showed
that any countable subgroup of $GL(n,R)$, the group of invertible $n \times n$ matrices over an arbitrary commutative ring $R$,
has FDC,~\cite{Guentner_Tessera_Yu2}.
Such a group can have infinite asymptotic dimension; for example,
the wreath product $\Z \wr \Z$ (this finitely generated group can be realized as a subgroup of $GL(2, \Z[t, t^{-1}])$).
The collection of countable groups with FDC contains groups with finite asymptotic dimension and has nice inheritance properties:
It is closed under subgroups, extensions, free products with amalgamation, HNN extensions and countable direct unions.
The FDC condition was introduced to study topological rigidity questions,~\cite{Guentner_Tessera_Yu1}. In this paper we focus on finite decomposition complexity as a coarse geometric invariant.

The definition of finite decomposition complexity is somewhat tricky to work with, so it is advantageous to develop tools for determining whether or not a metric space has FDC. In this paper we introduce some new tools for working with finite decomposition complexity (both the {\it strong} and {\it weak} forms), and try to give a feel for decomposition complexity by using these tools in several situations.
Motivated by the equivalent definitions for finite asymptotic dimension,
we provide analogous conditions that are equivalent to decomposition complexity, which we then use in the following two applications. 

Buyalo and Schroeder introduced the {\it hyperbolic dimension} of a metric space (Definition~\ref{def:hyperbolicdim}) to study quasi-isometric embedding properties of negatively curved spaces. 
Cappadocia introduced the related notion of the {\it weak hyperbolic dimension} of a metric space (Definition~\ref{def:weaklydoubling}).
Hyperbolic dimension is an upper bound for weak hyperbolic dimension.
We show in Corollary~\ref{cor:whyperdimdecompose} that a metric space with  weak hyperbolic dimension at most $n$ 
is $n$-{\it decomposable} (Definition~\ref{def:n-decompose}) over the collection of metric families with finite asymptotic dimension.
This implies that such a metric space has weak FDC; if $n\leq 1$, then it has FDC. 

In~\cite{Dadarlat_Guentner1}, Dadarlat and Guentner introduced the notion of a family of metric spaces that is {\it coarsely embeddable into Hilbert space}\footnote{Dadarlat and Guentner used the phrase ``equi-uniformly embeddable" instead of ``coarsely embeddable".} (Definition~\ref{def:coarse-embedding-Hilbert}). In Theorem~\ref{thm:decomp-coarse-embedding}, we show that if a metric family is $n$-decomposable over the collection of metric families that are coarsely embeddable into Hilbert space, then that metric family is also coarsely embeddable into Hilbert space. In other words, the collection of metric families that are coarsely embeddable into Hilbert space is stable under decomposition. This recovers the known fact that a metric space with (strong or weak) FDC is coarsely embeddable into Hilbert space.

Not all metric spaces satisfy the FDC condition. Clearly, any metric space that does not coarsely embed into Hilbert space will not have (strong or weak) FDC. Generalizing an example of Wu and Chen~\cite{WuChen}, we provide a tool that can be used to show that certain metric spaces do not have (strong or weak) FDC. In Theorem~\ref{thm:notweak FDC} it is shown that if a metric space $X$ admits a surjective {\it uniform expansion} and has weak finite decomposition complexity, then $X$ has finite asymptotic dimension. Thus, as explained in Example~\ref{exam:normed-linear}, any infinite-dimensional normed linear space cannot have (strong or weak) FDC because such a space has infinite asymptotic dimension and admits a surjective uniform expansion.

In the final section of this paper we recall some interesting open problems about decomposition complexity and suggest a few new ones.


\section{Decomposition Complexity}
\label{sec:DC}

Guentner, Tessera and Yu's concept of finite decomposition complexity was motivated by the following definition of finite asymptotic dimension.

\begin{definition}\label{def:asdim}
	Let $n$ be a non-negative integer. The metric space $(X,d)$ has {\it asymptotic dimension at most $n$}, $\asdim X \leq n$, if for every $r>0$ there exists a cover $\cU$ of $X$ such that 
	\begin{enumerate}
		\item[(i)] $\cU=\cU_0 \cup \cU_1 \cup \cdots \cup \cU_n$;
		\item[(ii)] each $\cU_i$, $0\leq i \leq n$, is {\it $r$-disjoint}, i.e., $d(U,V) > r$ for every $U\neq V$ in $\cU_i$; and
		\item[(iii)] $\cU$ is {\it uniformly bounded}, i.e., the {\it mesh} of $\cU$, $\mesh(\cU)=\sup\{\diam(U)~|~U \in \cU\}$, is finite.
	\end{enumerate}
	If no such $n$ exists, then $\asdim X = \infty$.
\end{definition}

In the above definition, note that the cover $\cU$ has {\it multiplicity} at most $n+1$, i.e., every point of $X$ is contained in at most $n+1$ elements of $\cU$. Also note that while $\cU$ is not required to be an open cover, if $\asdim X < \infty$, then one can always choose $\cU$ to be an open cover because of condition~(ii). 

The asymptotic dimension of a finitely generated group, $G$, is defined to be the asymptotic dimension of $G$ considered as a metric space with the word-length metric associated to any finite set of generators. This is well-defined since asymptotic dimension is a coarse invariant and any two finite generating sets for $G$ yield coarsely equivalent metric spaces. More generally, every countable group $G$ admits a proper left-invariant metric that is unique up to coarse equivalence. Thus, asymptotic dimension can also be used as a coarse invariant for countable groups. 

In order to generalize the definition of finite asymptotic dimension, it is useful to work with the notion of a {\it metric family}, a (countable) collection of metric spaces.
A single metric space is viewed as a metric family with one element.
A {\it subspace} of a metric family $\cX$ is a metric family $\cZ$ such that every element of $\cZ$ is a metric subspace of some element of $\cX$. For example, a cover $\cU$ of a metric space $X$ is a metric family, where each element of $\cU$ is given the subspace metric inherited from $X$, and $\cU$ is a subspace of the metric family~$\{ X \}$.

\begin{definition}\label{def:r-decompose}
	Let $r>0$ and $n$ be a non-negative integer. The metric family $\cX$ is {\em $(r,n)$-decomposable} over the metric family $\cY$, denoted $\cX \xrightarrow{(r,n)} \cY$, if for every $X$ in $\cX$, $X=X_0 \cup X_1 \cup \cdots \cup X_n$ such that for each $i$ 
	$$X_i=\bigsqcup_{r\text{-disjoint}}X_{ij}$$ 
where each $X_{ij}$ is in $\cY$. 
\end{definition}

\begin{definition}\label{def:n-decompose}
	Let $n$ be a non-negative integer, and let $\fC$ be a collection of metric families. The metric family $\cX$ is {\em $n$-decomposable} over $\fC$ if for every $r>0$ $\cX$ is $(r,n)$-decomposable over some metric family $\cY$ in $\fC$. 
\end{definition}

Following~\cite{Guentner_Tessera_Yu1}, we say that $\cX$ is {\em weakly decomposable} over $\fC$ if $\cX$ is $n$-decomposable over $\fC$ for some non-negative integer $n$, and $\cX$ is {\em strongly decomposable} over $\fC$ if $\cX$ is $1$-decomposable over $\fC$.

\begin{definition}
	A metric family $\cZ$ is {\em bounded} if the diameters of the elements of $\cZ$ are uniformly bounded, i.e., if $\sup\{\diam(Z)~|~Z \in \cZ\}<\infty$. The collection of all bounded metric families is denoted by $\fB$.
\end{definition}

\begin{example}\label{ex-asdim}
	Let $X$ be a metric space. The statement that the metric family $\{ X \}$ is $n$-decomposable over $\fB$ is equivalent to the statement that $\asdim(X) \leq n$.
\end{example}

The following definition is equivalent to Bell and Dranishnikov's definition of a collection of metric spaces having finite asymptotic dimension ``uniformly" (\cite[Section 1]{Bell_Dranish1}).

\begin{definition}\label{def:family-asdim}
	Let $n$ be a non-negative integer. The metric family $\cX$ has {\em asymptotic dimension at most $n$}, denoted $\asdim(\cX)\leq n$, if $\cX$ is $n$-decomposable over $\fB$.
\end{definition}

\begin{example}\label{exam:balls}
For each positive integer $n$,  let $\cX_n$ be the metric family of subsets of $\R^n$, with the Euclidean metric, consisting of open balls centered at the origin with positive integer radius.
Then $\asdim(\cX_n) =n$.
\end{example}

\begin{proof}
Since $\cX_n$ is a family of subspaces of $\R^n$, we have that  $\asdim(\cX_n) \leq \asdim(\R^n) =n$.

Suppose that $\asdim(\cX_n) = \ell < n$.
Then, for each integer $m \geq 1$, there exists a cover $\cU_m$, which can be assumed to be an open cover of the open ball $B_m(0)$, with multiplicity at most $\ell +1$, such that $\sup\{\mesh(\cU_m) ~|~ m \geq 1\} = D < \infty$.

Let $\epsilon > 0$. Choose an integer $k$ so that $k > D/\epsilon$.
For $\lambda > 0$ and $A \subset \R^n$ let $\lambda \, A = \{ \lambda a ~|~ a \in A \}$.
Then $\cU = \{ \tfrac{1}{k} \, U ~|~ U \in \cU_k \}$  is an open cover of $B_1(0)$ 
with $\mesh(\cU) < \epsilon$ and multiplicity at most $\ell +1$.
Hence the covering dimension of $B_1(0)$ is at most $\ell$, which
contradicts the fact that the covering dimension of $B_1(0)$ is $n$.
\end{proof}

\begin{definition}\label{def:FDC}
	Let $\fD$ be the smallest collection of metric families containing $\fB$ that is closed under strong decomposition, and let $\wD$ be the smallest collection of metric families containing $\fB$ that is closed under weak decomposition. A metric family in $\fD$ is said to have {\em finite decomposition complexity} (abbreviated to ``FDC"), and a metric family in $\wD$ is said to have {\em weak finite decomposition complexity} (abbreviated to ``weak FDC").
\end{definition}

Clearly, finite decomposition complexity implies weak finite decomposition complexity.
The converse is unknown. 

\begin{question}\cite[Question 2.2.6]{Guentner_Tessera_Yu2}
\label{doesweakimplystrong}
Does weak finite decomposition complexity imply finite decomposition complexity?
\end{question}

The base case of this question has an affirmative answer, namely, if $\asdim X < \infty$ then $X$ has FDC,~\cite[Theorem 4.1]{Guentner_Tessera_Yu2} (although even this case is difficult). Therefore, if $\fA$ is the collection of all metric families with finite asymptotic dimension, then we have the following sequence of inclusions of collections of metric families:
\begin{equation}\label{eq:asdim_has_FDC}
	\fA \subset \fD \subset \wD.
\end{equation}

	The collection of countable groups (considered as metric spaces with a proper left-invariant metric) in $\fD$ is quite large. It contains countable subgroups of $GL(n,R)$, where $R$ is any commutative ring, countable subgroups of almost connected Lie groups, hyperbolic groups and elementary amenable groups. It is also closed under subgroups, extensions, free products with amalgamation, HNN extensions and countable direct unions,~\cite{Guentner_Tessera_Yu2}. 
	
The FDC and weak FDC conditions have important topological consequences.
For example,
a finitely generated group with weak FDC satisfies the {\it Novikov Conjecture},
and a metric space with (strong) FDC and bounded geometry satisfies the {\it Bounded Borel Conjecture},~\cite{Guentner_Tessera_Yu1,Guentner_Tessera_Yu2}.
These results were obtained by studying certain {\it assembly maps} in $L$-theory and topological $K$-theory.
The assembly map in algebraic $K$-theory has been studied for groups with FDC by several authors, including Ramras, Tessera and Yu~\cite{Ramras-Tessera-Yu},
Kasprowski~\cite{Kasprowski}, and
Goldfarb~\cite{Goldfarb}.

\smallskip
	There is an equivalent description of FDC, and weak FDC, in terms of a {\it metric decomposition game},~\cite[Theorem 2.2.3]{Guentner_Tessera_Yu2}, that is useful for understanding the proofs of many of the inheritance properties mentioned above. 
The metric decomposition game has two players, a defender and a challenger. 
The game begins with a metric family $\cX = \cY_0$. 
On the first turn, the challenger declares a positive integer $r_1$ and the defender
must produce a $(r_1, n_1)$-decomposition of $\cY_0$ over a new metric family $\cY_1$.
On the second turn, the challenger declares a positive integer $r_2$  and the defender
must produce an $(r_2, n_2)$-decomposition of $\cY_1$ over a new metric family $\cY_2$.
The game continues in this manner, ending if and when the
defender produces a bounded family. In this case the defender has won.
A winning strategy is a set of instructions that, if followed by the defender,
will guarantee a win for any possible requests made by the challenger.
The family  $\cX$ has weak FDC if a  winning strategy exists and strong FDC if, additionally, the strategy
always allows for $n_j =1$ in the defender's response.

\smallskip		
	Next, we recall some terminology introduced in~\cite{Guentner_Tessera_Yu2} that generalizes basic notions from the coarse geometry of metric spaces to metric families. 
	
	Let $\cX$ and $\cY$ be metric families. A {\it map of families},  $F:\cX \to \cY$, is a collection of functions $F=\{f:X \to Y\}$, where $X \in \cX$ and $Y \in \cY$, such that every $X \in \cX$ is the domain of at least one $f$ in $F$. The {\it inverse image of $\cZ$ under $F$} is the subspace of $\cX$ given by $F^{-1}(\cZ)=\{ f^{-1}(Z) \; | \; Z\in \cZ, f\in F \}$.

\begin{definition}\label{def:coarse-embedding}
	A map of metric families, $F:\cX \to \cY$, is a {\em coarse embedding} if there exist non-decreasing functions $\delta, \rho: [0,\infty) \to [0,\infty)$, with $\lim_{t\to \infty}\delta(t)=\infty=\lim_{t\to \infty}\rho(t)$, such that for every $f:X \to Y$ in $F$ and every $x,y\in X$,
	\[ \delta\big(d_X(x,y)\big) \leq d_Y\big(f(x),f(y)\big) \leq \rho\big(d_X(x,y)\big). \]
\end{definition}

One can think of a coarse embedding of metric families as a collection of ``uniform" coarse embeddings, in the sense that they have a common $\delta$ and $\rho$. The easiest example of a coarse embedding of metric families is the inclusion of a subspace $\cZ$ of $\cY$ into $\cY$.

\begin{definition}\label{def:coarse-equiv}
	A map of metric families, $F:\cX \to \cY$, is a {\em coarse equivalence} if for each $f:X \to Y$ in $F$ there is a map $g_f:Y\to X$ such that: 
	\begin{enumerate}
		\item[(i)] the collection $G=\{g_f\}$ is a coarse embedding from $\cY$ to $\cX$; and 
		\item[(ii)] the composites $f\circ g_f$ and $g_f \circ f$ are {\em uniformly close} to the identity maps $\id_Y$ and $\id_X$, respectively, in the sense that there is a constant $C>0$ with
	\[ d_Y\big(y,f\circ g_f(y)\big)\leq C \text{ and } d_X\big(x,g_f\circ f(y)\big)\leq C, \]
for every $f:X \to Y$ in $F$, $x\in X$, and $y\in Y$.
	\end{enumerate}
\end{definition}

\begin{definition}\label{def:closed-under-embeddings}
	A collection of metric families, $\fC$, is {\em closed under coarse embeddings} if every metric family $\cX$ that coarsely embeds into a metric family $\cY$ in $\fC$ is also a metric family in $\fC$.
\end{definition}

Guentner, Tessera and Yu proved that $\fD$ and $\wD$ are each closed under coarse embeddings~\cite[Coarse Invariance 3.1.3]{Guentner_Tessera_Yu2}. It is straightforward to check that the following collections of metric families are also closed under coarse embeddings.

\begin{example}\label{ex:ce} 
Collections of metric families that are closed under coarse embeddings:
	\begin{enumerate}
		\item $\fB$, the collection of bounded metric families.
		\item $\fA$, the collection of metric families with finite asymptotic dimension.
		\item $\fA_n$, the collection of metric families with asymptotic dimension at most $n$.
		\item $\fH$, the collection of metric families that are {\it coarsely embeddable into Hilbert space} (see Definition~\ref{def:coarse-embedding-Hilbert} below).
		
	\end{enumerate}
\end{example}

The following is similar to \cite[Coarse Invariance 3.1.3]{Guentner_Tessera_Yu2}.

\begin{theorem}\label{thm:cembed-decomp}
	Let $\cX$ and $\cY$ be metric families, and let $\fC$ be a collection of metric families that is closed under coarse embeddings. If $\cX$ coarsely embeds into $\cY$ and $\cY$ is $n$-decomposable over $\fC$, then $\cX$ is $n$-decomposable over $\fC$. In particular, if $\cX$ is coarsely equivalent to $\cY$, then $\cX$ is $n$-decomposable over $\fC$ if and only if $\cY$ is $n$-decomposable over~$\fC$.
\end{theorem}

\begin{proof}
	Let $F:\cX \to \cY$ be a coarse embedding, and let $r>0$ be given. We must find a metric family $\cX'$ in $\fC$ such that $\cX$ is $(r,n)$-decomposable over $\cX'$. Since $\cY$ is $n$-decomposable over $\fC$, there is a metric family $\cY'$ such that $\cY$ is $(\rho(r),n)$-decomposable over $\cX'$, where $\rho$ is as in Definition~\ref{def:coarse-embedding}. It is straightforward to show that $\cX$ is $(r,n)$-decomposable over $\cX'=F^{-1}(\cY')$. Note that $F$ restricts to a coarse embedding from $F^{-1}(\cY')$ to $\cY'$. Since $\fC$ is closed under coarse embeddings, we are done.
\end{proof}


The following observation about decomposition is useful.

\begin{remark}\label{rem:finite-asdim}
\label{rm:decomposition}
If $\cX$, $\cY$, and $\cZ$ are metric families and
$\cX \xrightarrow{(r,m)} \cY \xrightarrow{(s,n)} \cZ$, then
$\cX \xrightarrow{(t, p)} \cZ$  where $t=\min(r,s)$ and $p=(m+1)(n+1)-1$. In particular, this shows that if $\cX$ is $m$-decomposable over $\fA_n$, then $\cX$ has asymptotic dimension at most $(m+1)(n+1)-1$.
\end{remark}

Let $(X, d_X)$ be a metric space and $\lambda > 1$. 
A  {\it uniform expansion of $X$ with expansion factor $\lambda$}  is  a map $T \colon X \rightarrow X$ such that $d_X(T(x), T(y)) = \lambda \, d_X(x,y)$ for all $x, y \in X$.
The following proposition,
generalizing
\cite[Example 2.2]{WuChen},
can be used to show that certain spaces do {\it not} have weak finite decomposition complexity.
	
\begin{theorem}
\label{thm:notweak FDC}
 Let $(X,d_X)$ be a metric space that admits a surjective uniform expansion.
 If $X$ has weak finite decomposition complexity then $X$ has finite asymptotic dimension.
\end{theorem}

\begin{proof}
Assume that the metric space $(X,d_X)$ has weak FDC and that $T \colon X \rightarrow X$ is a surjective uniform expansion of $X$ with expansion factor $\lambda > 1$.
By the analog of \cite[Theorem 2.4]{Guentner_Tessera_Yu1} for weak FDC (while \cite[Theorem 2.4]{Guentner_Tessera_Yu1} is stated for FDC, the proof there readily adapts to weak FDC),
there exists a finite sequence $(r_i, \, n_i)$, $i=1,\ldots,m$,  where each $r_i >0$ and the $n_i$'s are positive integers, together with metric families
$\cY_i$, $i=1,\ldots,m$, such that
\[
X \xrightarrow{(r_1,n_1)} \cY_1 \xrightarrow{(r_2,n_2)} \cY_2  \xrightarrow{\phantom{(r_2,n_2)}} \cdots \xrightarrow{(r_m,n_m)} \cY_m
\]
and $\cY_m \in \fB$.
(This is one winning round of the ``decomposition game.'')
By Remark \ref{rm:decomposition}, we have $X \xrightarrow{(r,n)} \cY_m$,
where $r=\min\{r_1, \ldots, r_m\}$ and $n =(n_1 +1)(n_2 +1) \cdots (n_m+1) \, -1$.
For any positive integer $k$, let $T^k (\cY_m) = \big\{ T^k(Y) ~|~ Y \in \cY_m \big\}$.
Notice that $T^k (\cY_m) \in \fB$.
Since $T$ is surjective, $T^k(X) = X$. 
Hence, $$\{X\} = \big\{T^k(X)\big\}  \xrightarrow{(\lambda^k r, \, n)} T^k (\cY_m).$$
Since $\lambda > 1$, we have $\lambda^k r \rightarrow \infty$ as $k \rightarrow \infty$. 
It follows that $\{X\}$ is $n$-decomposable over $\fB$. That is, $X$ has finite asymptotic dimension.
\end{proof}
	
\begin{example}\label{exam:normed-linear}
Let $(V, \|\cdot\|)$ be any infinite-dimensional normed linear space. 
Then $T(x) = 2x$ is a uniform expansion of $V$,  with expansion factor $2$,  where the metric is $d(x,y) = \| x-y\|$.
Clearly, $T$ is surjective.   Note that any real $n$-dimensional vector subspace of $V$ has asymptotic dimension $n$ and so 
$V$ has infinite asymptotic dimension. 
It follows from Theorem~\ref{thm:notweak FDC} that $(V,d)$ cannot have weak FDC.
\end{example}

\begin{example}
The condition in Theorem~\ref{thm:notweak FDC} that the uniform expansion $T$ is surjective cannot be omitted.
Consider $X = \bigoplus^\infty_{i=1} \Z$ with the proper metric 
$d_X\big( (x_i), (y_i)\big) = \sum^\infty_{i=1}  i\cdot| x_i - y_i|$.   
Observe that $T\big( (x_i)\big) = (2 x_i)$ is a uniform expansion of $(X,d_X)$, with expansion factor $2$,
but $T$ is not surjective.
Although $(X,d_X)$ has infinite asymptotic dimension, it has FDC
(see \cite[Example 2.5]{Guentner_Tessera_Yu1})
 and hence weak FDC.
 
Now consider $Y=\bigoplus^\infty_{i=1} \R$ equipped with the metric $d_Y\big( (x_i), (y_i)\big) = \sum^\infty_{i=1}  i\cdot |x_i - y_i|$. Then $X$ is a metric subspace of $Y$, and for each $n\in \bN$, the subspace $\bigoplus^n_{i=1} \Z$ of $X$ is coarsely equivalent to the subspace $\bigoplus^n_{i=1} \R$ of $Y$. Nevertheless, $X$ is {\it not} coarsely equivalent to $Y$, since $X$ does not have weak FDC by Example~\ref{exam:normed-linear}. 
\end{example}

If $\cX$ is a metric family and  $N = \sup\{ \asdim(X) ~|~ X \in \cX\}$, 
then clearly $N \leq \asdim(\cX)$.  
Equality often does {\it not} hold. For example,
consider the space
$Z = \bigoplus^\infty_{i=1} \R$ with the Euclidean metric and the metric family $\cX= \{ B_r(0) ~|~r = 1, 2, \ldots \}$ of open balls in $Z$. For each positive integer $n$, let $\cX_n= \{ B_r(0)\cap \R^n ~|~r = 1, 2, \ldots \}$, where $\R^n$ denotes the metric subspace $\bigoplus^n_{i=1} \R$ in $Z$.  Then, by Example~\ref{exam:balls}, $\asdim(\cX_n)=n$. Therefore, $n=\asdim(\cX_n)\leq \asdim(\cX)$ for every positive integer $n$, and so $\asdim(\cX) = \infty$, whereas $\asdim(B_r(0)) = 0$ for each $r$.

However, as was pointed out to us by Daniel Kasprowski, for every countable discrete group $G$ equipped with a proper left-invariant metric, the family of finite subgroups of $G$ does have asymptotic dimension zero as a metric family. The following proposition is a generalization of this fact.

\begin{proposition}
	Let $G$ be a countable discrete group equipped with a proper left-invariant metric $d$. Let $\cF$ be a non-empty collection of subgroups of $G$ that is closed under taking subgroups.  If $\asdim(H)\leq k$ for every $H$ in $\cF$, where $H$ is considered as a metric subspace of $G$, then $\asdim(\cF)\leq k$.
\end{proposition}

\begin{proof}
	Let $r>0$ be given. For each $H$ in $\cF$, let $S_H$ be the subgroup of $H$ generated by $H \cap B_r(e)$, where $e$ is the identity element of $G$. Let $\cU_H$ be the set of left cosets of $S_H$ in $H$. Then, for every $x,y\in H$, 
\[ d(x,y)\leq r  \;\; \Leftrightarrow  \;\;  x^{-1}y\in B_r(e)  \;\;   \Leftrightarrow   \;\;  x^{-1}y\in S_H. \]
Thus, $\cU_H$ is an $r$-disjoint, $0$-dimensional cover of $H$. Let $\cY$ be the metric family $\bigcup_{H\in \cF} \cU_H$. Then, $\cF$ is $(r,0)$-decomposable over $\cY$. Since $\cU_H$ is coarsely equivalent to $\{ S_H \}$, it follows that $\cY$ is coarsely equivalent to $\{ S_H \; | \; H\in \cF\}$, which is a finite set since $d$ is a proper metric. Therefore, $\asdim(\cY)=\asdim\big( \{ S_H \; | \; H\in \cF\} \big) \leq k$. Thus, $\cF$ is $0$-decomposable over $\fA_k$, the collection of all metric families that have asymptotic dimension at most $k$. It follows from Remark~\ref{rem:finite-asdim} that $\asdim(\cF)\leq k$. 
\end{proof}

\section{Equivalent Definitions of Decomposability}	
	In this section we provide three alternative definitions for a metric family $\cX$ to be $n$-decomposable over a collection of metric families $\fC$. We show that they are all equivalent to Definition~\ref{def:n-decompose}, provided $\fC$ is closed under coarse embeddings.  When $\fC=\fB$ (the collection of all bounded metric families) and $\cX$ consists of a single metric space, each of our definitions reduces to one of the standard definitions for finite asymptotic dimension.

	Recall that the multiplicity of a covering $\cU$ of a metric space $X$ is the largest integer $m$ such that every point of $X$ is contained in at most $m$ elements of $\cU$. Given $d>0$, the {\it $d$-multiplicity} of $\cU$ is the largest integer $m$ such that every open $d$-ball, $B_d(x)$,  in $X$ is contained in at most $m$ elements of $\cU$. The {\it Lebesgue number} of $\cU$, $L(\cU)$, is at least $\lambda>0$ if every $B_\lambda(x)$ in $X$ is contained in some element of $\cU$. A {\it uniform simplicial complex} $K$ is a simplicial complex equipped with the $\ell^1$-metric. That is, every element $x\in K$ can be uniquely written as $x=\sum_{v\in K^{(0)}}x_v\cdot v$, where $K^{(0)}$ is the vertex set of $K$, each $x_v\in [0,1]$, $x_v=0$ for all but finitely many $v\in K^{(0)}$, and $\sum_{v\in K^{(0)}}x_v=1$. Then the $\ell^1$-metric is defined by $d^1(x,y)=\sum_{v\in K^{(0)}}| x_v-y_v |$. The {\it open star} of a vertex $v\in K^{(0)}$ is the set $\st(v)=\{ x\in K\;|\; x_v\neq 0 \}$. If there exists an integer $m$ such that for every $x\in K$ the set $\{v\in K^{(0)}\;| \;x_v\neq 0\}$ has cardinality at most $m$, then the {\it dimension of $K$}, $\dim(K)$, is at most $m$. If no such $m$ exists, then $\dim(K)=\infty$.
	
	In what follows, let $\cX=\{X_\alpha \, | \, \alpha \in I\}$ be a metric family, where $I$ is a countable indexing set, and let $\fC$ be a collection of metric families. Let $n$ be a non-negative integer.

\begin{conditionA}\label{A}
	For every $d>0$, there exists a cover $\cV_\alpha$ of $X_\alpha$, for each $\alpha \in I$, such that:
	\begin{enumerate}	
	 	\item[(i)] the $d$-multiplicity of $\cV_\alpha$ is at most $n+1$ for every $\alpha \in I$; and
		\item[(ii)] $\bigcup_{\alpha\in I}\cV_\alpha$ is a metric family in $\fC$.
	\end{enumerate}
\end{conditionA}

\begin{conditionB}\label{B}
	For every $\lambda>0$, there exists a cover $\cU_\alpha$ of $X_\alpha$, for each $\alpha \in I$, such that:
	\begin{enumerate}	
	 	\item[(i)] the multiplicity of $\cU_\alpha$ is at most $n+1$ for every $\alpha \in I$;
		\item[(ii)] the Lebesgue number $L(\cU_\alpha)\geq \lambda$ for every $\alpha \in I$; and
		\item[(iii)] $\bigcup_{\alpha\in I}\cU_\alpha$ is a metric family in $\fC$.
	\end{enumerate}
\end{conditionB}

\begin{conditionC}\label{C}
	For every $\varepsilon>0$, there exists a uniform simplicial complex $K_\alpha$ and an $\varepsilon$-Lipschitz map $\varphi_\alpha:X_\alpha \to K_\alpha$, for each $\alpha \in I$, such that:
	\begin{enumerate}	
	 	\item[(i)] $\dim(K_\alpha)\leq n$ for every $\alpha \in I$; and
		\item[(ii)] $\bigcup_{\alpha\in I}\big\{ \varphi_\alpha^{-1}\big(\st(v)\big) \; \big| \; v \in K_\alpha^{(0)} \big\}$ is a metric family in~$\fC$.
	\end{enumerate}
\end{conditionC}

\begin{proposition}\label{prop:equiv}
	Let $\cX$ be a metric family and $\fC$ be a collection of metric families that is closed under coarse embeddings. Then Conditions~\hyperref[A]{(A)} and~\hyperref[B]{(B)} are each equivalent to Definition~\ref{def:n-decompose}.
\end{proposition}

\begin{proof}
	For notational convenience we prove the proposition when $\cX$ consists of a single metric space $X$. The proof for a general metric family is a straightforward generalization of this case. 
	
	Suppose that $X$ is $n$-decomposable over $\fC$. Let $d>0$ be given. Then, there is a metric family $\cY$ in $\fC$ and a decomposition $X=X_0 \cup X_1 \cup \cdots \cup X_n$ such that, for each $i$
	$$X_i=\bigsqcup_{2d\text{-disjoint}}X_{ij}$$  
where each $X_{ij}$ is in $\cY$. Thus, the cover $\cV=\{ X_{ij} \}$ of $X$ is a subspace of $\cY$ and has $d$-multiplicity less than or equal to $n+1$. Since $\fC$ is closed under coarse embeddings, $\cV$ is also in $\fC$ and Condition~\hyperref[A]{(A)} is satisfied. 

	Suppose that Condition~\hyperref[A]{(A)} is satisfied for $n$ with respect to $\fC$ and let $\lambda>0$ be given. There exists a cover $\cV$ of $X$ that is a metric family in $\fC$ and has $\lambda$-multiplicity less than or equal to $n+1$. Let $\cU=\big\{V^\lambda \; \big| \; V\in \cV  \big\}$, where $V^\lambda$ is the set of points in $X$ whose distance from $V$ is at most $\lambda$. Then the Lebesgue number $L(\cU)\geq \lambda$.  Given $x \in X$, the ball of radius $\lambda$ around $x$ intersects at most $n+1$ elements of $\cV$, since the $\lambda$-multiplicity of $\cV$ is at most $n+1$. This implies that at most $n+1$ elements of $\cU$ contain $x$, i.e., the multiplicity of $\cU$ is at most $n+1$. Since $\cU$ is coarsely equivalent to $\cV$ and $\fC$ is closed under coarse embeddings (and hence under coarse equivalences), Condition~\hyperref[B]{(B)} is satisfied.
	
	Suppose that Condition~\hyperref[B]{(B)} is satisfied for $n$ with respect to $\fC$ and let $r>0$ be given. We follow an argument analogous to the one in \cite[Theorem 9]{Grave} to show that Condition~\hyperref[B]{(B)} implies Definition~\ref{def:n-decompose}. There exists a cover $\cU$ of $X$ such that $\cU$ has multiplicity at most $n+1$, $L(\cU)\geq (n+1)r$, and $\cU$ is in $\fC$. Given $d>0$ and $U \subset X$, let ${\rm Int}_d(U)=\big\{x \in X \;\big|\; B_d(x) \subset U \big\}$. Note that if $d_1 \leq d_2$, then ${\rm Int}_{d_2}(U) \subseteq {\rm Int}_{d_1}(U)$. Also note that if $a\in {\rm Int}_{d}(U) \cap {\rm Int}_{d}(V)$, then $a \in {\rm Int}_{d}(U \cap V)$. Now, for each $i\in \{0,\dots,n\}$, define
\renewcommand{\arraystretch}{2}
\[\begin{array}{rcl}
	\cU_i & = & \big\{U_0 \cap U_1 \cap \cdots \cap U_i \; \big|\; U_0,U_1,\dots,U_i \in \cU \text{ are distinct} \big\} \\
	S_i & = & \bigcup_{U \in \cU_i} {\rm Int}_{(n+1-i)r}(U)  \\
	X_i & = & \bigsqcup_{U \in \cU_i} {\rm Int}_{(n+1-i)r}(U) \smallsetminus S_{i+1}.
\end{array}\]
Since $\cU$ has multiplicity at most $n+1$ and has Lebesgue number $L(\cU)\geq (n+1)r$, it follows that $X=X_0 \cup X_1 \cup \cdots \cup X_n$. Furthermore, since each ${\rm Int}_{(n+1-i)r}(U) \smallsetminus S_{i+1}$ is contained in some element of $\cU$ and $\fC$ is closed under coarse embeddings, the metric family
$$\big\{{\rm Int}_{(n+1-i)r}(U) \smallsetminus S_{i+1} \;\big|\; 0\leq i \leq n \text{ and } U\in \cU_i \big\}$$
is in $\fC$. It remains to show that in fact each $X_i$ is an $r$-disjoint union. We do this by contradiction. Given $i$, suppose that ${\rm Int}_{(n+1-i)r}(U) \smallsetminus S_{i+1} \neq {\rm Int}_{(n+1-i)r}(V) \smallsetminus S_{i+1}$, where $U=U_0 \cap U_1 \cap \cdots \cap U_i$ and $V=V_0 \cap V_1 \cap \cdots \cap V_i$, and that there exist $a \in {\rm Int}_{(n+1-i)r}(U) \smallsetminus S_{i+1}$ and $b \in {\rm Int}_{(n+1-i)r}(V) \smallsetminus S_{i+1}$ with $d(a,b)\leq r$. Then $a\in \big({\rm Int}_{(n+1-i)r}(V)\big)^r$ and $b\in \big({\rm Int}_{(n+1-i)r}(U)\big)^r$. Notice that for each natural number $k$, the $r$-neighborhood 
$$\big({\rm Int}_{(k+1)r}(U)\big)^r= \big\{ y\in X \; \big| \; \exists \; x \text{ such that } B_{(k+1)r}(x)\subset U \text{ and } d(y,x)\leq r \big\}$$
is contained in ${\rm Int}_{kr}(U)=\big\{y \in X \;\big|\; B_{kr}(y) \subset U \big\}$. Therefore, $a \in {\rm Int}_{(n-i)r}(V)$ and $b \in {\rm Int}_{(n-i)r}(U)$. Thus, $a,b\in {\rm Int}_{(n-i)r}(U) \cap {\rm Int}_{(n-i)r}(V)\subset {\rm Int}_{(n-i)r}(U\cap V)$. Since ${\rm Int}_{(n+1-i)r}(U) \smallsetminus S_{i+1} \neq {\rm Int}_{(n+1-i)r}(V) \smallsetminus S_{i+1}$, the set $\{U_0,\dots,U_i,V_0,\dots,V_i  \}$ has at least $i+2$ elements, which implies that $a$ and $b$ are both in ${\rm Int}_{(n-i)r}(U\cap V)={\rm Int}_{(n-i)r}(U_0 \cap \cdots \cap U_i \cap V_0 \cap \cdots \cap V_i) \subset S_{i+1}$. But this contradicts the assumption that $a \in {\rm Int}_{(n+1-i)r}(U) \smallsetminus S_{i+1}$ and $b \in {\rm Int}_{(n+1-i)r}(V) \smallsetminus S_{i+1}$. Therefore, $X$ is $n$-decomposable over $\fC$.
\end{proof}

\begin{lemma}\label{lem:simplicial_constant}
	Let $n$ be a non-negative integer. Then for every uniform simplicial complex $K$ with $\dim(K)\leq n$, the cover $\cV_K$ of $K$ consisting of the open stars of vertices in $K$ has Lebesgue number $L(\cV_K)\geq \frac{1}{n+1}$.
\end{lemma}

\begin{proof}
	Let $K$ be a uniform simplicial complex of dimension $n$, and let $x \in K$ be given. Then $x=\sum_{v \in K^{(0)}}x_v\cdot v$, where there are at most $n+1$ vertices $v$ with $x_v\neq 0$ and $\sum_{v \in K^{(0)}}x_v=1$. There is a $v \in K^{(0)}$ with $x_v\geq \frac{1}{n+1}$. If $y\in K$ is not in the open star of $v$, $\st(v)$, then $y_v=0$ and $d^1(x,y) \geq \frac{1}{n+1}$. Therefore, the open ball of radius $\frac{1}{n+1}$ centered at $x$ is completely contained in $\st(v)$. Thus, the cover $\cV_K$ of $K$ consisting of the open stars of vertices in $K$ has Lebesgue number $L(\cV_K)\geq \frac{1}{n+1}$.
\end{proof}

\begin{proposition}\label{prop:epsilon_decomp}
	Let $\cX$ be a metric family and $\fC$ be a collection of metric families that is closed under coarse embeddings. Then Condition~\hyperref[C]{(C)} is equivalent to Definition~\ref{def:n-decompose}.
\end{proposition}
 
\begin{proof}
	For notational convenience we prove the proposition when $\cX$ consists of a single metric space $X$. The proof for a general metric family is a straightforward generalization of this case. 
	
	By Proposition~\ref{prop:equiv}, it suffices to prove that Condition~\hyperref[C]{(C)} is equivalent to Condition~\hyperref[B]{(B)}. We follow~\cite[Assertion 2]{Bell_Dranish1} to show that Condition~\hyperref[C]{(C)} implies Condition~\hyperref[B]{(B)}, and we follow~\cite[Theorem 1]{Bell_Dranish2} to show that Condition~\hyperref[B]{(B)} implies Condition~\hyperref[C]{(C)}.
	
	Assume that $X$ satisfies Condition~\hyperref[C]{(C)} for $n$ with respect to $\fC$. Let $r>0$ be given. Then, by Lemma~\ref{lem:simplicial_constant}, there is a uniform simplicial complex $K$ of dimension $n$ and a $\frac{1}{(n+1)r}$-Lipschitz map $\varphi: X \to K$. Since $\dim(K)=n$, the cover $\cU=\big\{ \varphi^{-1}\big(\st(v)\big) \; \big| \; v\in K^{(0)}\big\}$ of $X$  has multiplicity at most $n+1$ and Lebesgue number $L(\cU)>r$. By assumption, the metric family $\big\{ \varphi^{-1}\big(\st(v)\big) \; \big| \; v\in K^{(0)} \big\}$ is in $\fC$. Thus, $X$ satisfies Condition~\hyperref[B]{(B)}.

	Now assume that $X$ satisfies Condition~\hyperref[B]{(B)} for $n$ with respect to $\fC$. Let $\varepsilon>0$ be given. Then there is a cover $\cU$ of $X$ that is a metric family in $\fC$, has multiplicity at most $n+1$ and has Lebesgue number $L(\cU)\geq \lambda=\frac{(2n+2)(2n+3)}{\varepsilon}$. Note that, because $L(\cU)>0$ and $\fC$ is closed under coarse embeddings, we may additionally assume, without loss of generality, that $\cU$ is an open covering of $X$. For each $U\in \cU$, define $\varphi_U: X \to [0,1]$ by
	\[ \varphi_U(x)=\dfrac{d(x,U^c)}{\sum_{V \in \;\cU}d(x,V^c)} \]
where $U^c$ is the complement of $U$ in $X$. Let $K={\rm Nerve}(\cU)$ equipped with the uniform metric. Since the multiplicity of $\cU$ is at most $n+1$, $\dim(K) \leq n$. Define the map $\varphi:X \to K$ by
	\[ \varphi(x)=\sum_{U \in \; \cU}\varphi_U(x) \cdot [U] \]
where $[U]$ denotes the vertex of $K$ defined by $U$. Note that given a vertex $[V]$ in $K$, $\varphi^{-1}(\st([V]))\subset V$, since $\varphi(x)$ is in the open star of $[V]$ if and only if $\varphi_V(x) \neq 0$, and this implies that $x$ is in $V$. Therefore, the metric family $\big\{ \varphi^{-1}(\st([V])) \; \big| \; [V] \in K^{(0)} \big\}\subset \cU$ is in $\fC$ since $\fC$ is closed under coarse embeddings.

	It remains to show that $\varphi$ is $\varepsilon$-Lipschitz. Since $L(\cU)\geq \lambda$, it follows that $\sum_{V \in \;\cU}d(x,V^c)\geq \lambda$. Also note that for every $x,y \in X$ and $U \in \cU$, the triangle inequality implies
	\[ \big| d(x,U^c) - d(y,U^c)\big| \leq d(x,y). \]
Thus, {\small
\renewcommand{\arraystretch}{3}
\[\begin{array}{rcl}
	\big|\varphi_U(x)-\varphi_U(y)\big| & = & \left| \dfrac{d(x,U^c)}{\sum_{V \in \;\cU}d(x,V^c)} - \dfrac{d(y,U^c)}{\sum_{V \in \;\cU}d(y,V^c)}  \right| \\
	& \leq & \dfrac{|d(x,U^c) - d(y,U^c)|}{\sum_{V \in \;\cU}d(x,V^c)} + \left| \dfrac{d(y,U^c)}{\sum_{V \in \;\cU}d(x,V^c)} - \dfrac{d(y,U^c)}{\sum_{V \in \;\cU}d(y,V^c)} \right| 
\end{array} \] }

\noindent
which is less than or equal to
{\small
\renewcommand{\arraystretch}{2}
	\[ \dfrac{d(x,y)}{\sum_{V \in \;\cU}d(x,V^c)} +  \dfrac{d(y,U^c)}{\Big(\sum_{V \in \;\cU}d(x,V^c)\Big) \Big(\sum_{V \in \;\cU}d(y,V^c)\Big)}\cdot \sum_{V \in \;\cU} \big|d(x,V^c)-d(y,V^c)\big|,\]}

\noindent
which is less than or equal to
{\small 
\renewcommand{\arraystretch}{2}
\[\begin{array}{rcl}
	\frac{1}{\lambda}\, d(x,y) + \frac{1}{\lambda} \Big(\sum_{V \in \;\cU} \big|d(x,V^c)-d(y,V^c)\big|\Big) & \leq & \frac{1}{\lambda}\, d(x,y) + \frac{1}{\lambda}\, 2(n+1)\, d(x,y)  \\
	& = & \frac{1}{\lambda}\, (2n+3)\,  d(x,y).
\end{array}\]
}
\noindent
Therefore,
{\small 
$$ d^1(\varphi(x),\varphi(y)) \, = \sum_{U \in \;\cU} \big|\varphi_U(x)-\varphi_U(y)\big| \,\leq\, 2(n+1)\, \left(\frac{1}{\lambda}\, (2n+3)\, d(x,y)\right) \,=\, \varepsilon \, d(x,y).
$$
}
\noindent
This completes the proof.
\end{proof}

%

The equivalent definitions for decomposability give us more tools to work with. For instance, consider the collection of metric families that are {\it coarsely embeddable into Hilbert space}, defined below. The notion of a metric family that is coarsely embeddable into Hilbert space was introduced by Dadarlat and Guentner in~\cite{Dadarlat_Guentner1}, although they called it a ``family of metric spaces that is equi-uniformly embeddable."  

\begin{definition}\label{def:coarse-embedding-Hilbert}
	A metric family $\cX=\{X_\alpha \, | \, \alpha \in I\}$ is {\em coarsely embeddable into Hilbert space} if there is a family of Hilbert spaces $\cH=\{ H_\alpha  \, | \, \alpha \in I\}$ and a map of metric families $F=\{F_\alpha:X_\alpha \to H_\alpha \, | \, \alpha \in I\}$ such that $F:\cX \to \cH$ is a coarse embedding. The collection of all metric families that are coarsely embeddable into Hilbert space is denoted by $\fH$.
\end{definition}

	In~\cite{Dadarlat_Guentner2}, Dadarlat and Guentner proved the following.

\begin{proposition}\label{prop:DG}\cite[Proposition 2.3]{Dadarlat_Guentner2}
	 A metric family $\cX=\{X_\alpha \, | \, \alpha \in I\}$ is in $\fH$ if and only if for every $R>0$ and $\varepsilon>0$ there exists a family of Hilbert spaces $\cH=\{ H_\alpha  \, | \, \alpha \in I\}$ and a map of metric families $\xi=\{\xi_\alpha:X_\alpha \to H_\alpha \, | \, \alpha \in I\}$ such that
	 \begin{enumerate}
	 	\item[(i)] $\| \xi_\alpha(x) \|=1$, for all $x\in X_\alpha$ and $\alpha\in I$;
		\item[(ii)] $\forall \alpha\in I$, $\forall x,x'\in X_\alpha$, $d_\alpha(x,x')\leq R \; \Rightarrow \; \| \xi_\alpha(x)-\xi_\alpha(x') \|\leq \varepsilon$;
		\item[(iii)] $\lim_{S \to \infty}\sup_{\alpha \in I}\sup \left\{ \big|\langle \xi_\alpha(x), \xi_\alpha(x') \rangle \big| : d_\alpha(x,x')\geq S, \; x,x'\in X_\alpha \right\}=0$. \qed
	 \end{enumerate}
\end{proposition}

	 Proposition~\ref{prop:epsilon_decomp} enables us to make use of Dadarlat and Guentner's work to prove the following theorem.

\begin{theorem}\label{thm:decomp-coarse-embedding}
	The collection $\fH$ of metric families that are coarsely embeddable into Hilbert space is stable under weak decomposition. That is, if a metric family $\cX$ is $n$-decomposable over $\fH$, then $\cX$ is in $\fH$.
\end{theorem}

\begin{proof}
	In light of Proposition~\ref{prop:epsilon_decomp}, all of the ingredients for the proof of this theorem are contained in~\cite{Dadarlat_Guentner2}. The argument is organized as follows.
	 
	Let $\cX=\{X_\alpha \, | \, \alpha \in I\}$ be a metric family that is $n$-decomposable over $\fH$. We will use Proposition~\ref{prop:DG} to prove that $\cX$ is in $\fH$. Let $R>0$ and $\varepsilon>0$ be given. Since $\cX$ is $n$-decomposable over $\fH$, the proof of Proposition~\ref{prop:epsilon_decomp} implies that for each $\alpha\in I$ there is a cover $\cU_\alpha=\{ U_{\alpha,j} \}_{j \in J_\alpha}$ of $X_\alpha$, for some indexing set $J_\alpha$, and a family of maps $\varphi_\alpha=\{ \varphi_{\alpha,j}:X_\alpha \to [0,1] \; | \; j\in J_\alpha \}$ such that
	\begin{enumerate}
		\item[(a)] $\sum_{j\in J_\alpha} \varphi_{\alpha,j}(x)=1$, for all $x\in X_\alpha$;
		\item[(b)] $\varphi_{\alpha,j}(x)=0$ if $x \notin U_{\alpha,j}$;
		\item[(c)] $\forall x,y \in X_\alpha$, $d_\alpha(x,y)\leq R\; \Rightarrow \; \sum_{j\in J} \big|\varphi_{\alpha,j}(x)-\varphi_{\alpha,j}(y)\big| \leq \frac{\varepsilon^2}{4}$ ;
		\item[(d)] the metric family $\{U_{\alpha,j} \, | \, \alpha\in I, j\in J_\alpha \}$ is in $\fH$.
	\end{enumerate}

The metric family $\{U_{\alpha,j}^R~|~\alpha\in I, j\in J_\alpha \}$, where $U_{\alpha,j}^R=\{ x\in X_\alpha~|~d_\alpha(x,U_{\alpha,j})\leq R \}$, is coarsely equivalent to the metric family $\{U_{\alpha,j} \; | \; \alpha\in I, j\in J_\alpha \}$. Therefore, since $\{U_{\alpha,j} \; | \; \alpha\in I, j\in J_\alpha \}$ is in $\fH$, so is $\{U_{\alpha,j}^R \; | \; \alpha\in I, j\in J_\alpha \}$. By Proposition~\ref{prop:DG}, there exists a family of Hilbert spaces $\cH=\{ H_{\alpha,j}  \, | \, \alpha \in I, j\in J_\alpha\}$ and a map of metric families $\xi=\{\xi_{\alpha,j}:U_{\alpha,j}^R \to H_{\alpha,j} \, | \, \alpha \in I, j\in J_\alpha\}$ satisfying
	 \begin{enumerate}
	 	\item[(i)] $\| \xi_{\alpha,j}(x) \|=1$, for all $x\in U_{\alpha,j}^R$;
		\item[(ii)] $\sup \big\{ \| \xi_{\alpha,j}(x)-\xi_{\alpha,j}(y)\| : d_\alpha(x,y)\leq R, x,y\in U_{\alpha,j}^R\big\} \leq \varepsilon/2$, for all $\alpha\in I$, $j\in J_\alpha$;
		\item[(iii)] $\lim_{S \to \infty}\sup_{\alpha \in I, j\in J_\alpha}\sup \left\{ \big|\langle \xi_{\alpha,j}(x), \xi_{\alpha,j}(y) \rangle \big| : d_\alpha(x,y)\geq S, \; x,y\in U_{\alpha,j}^R \right\}=0$.
	 \end{enumerate}

	For each $\alpha\in I$, extend $\xi_{\alpha,j}$ to all of $X_\alpha$ by setting $\xi_{\alpha,j}(x)=0$ if $x\in X_\alpha \smallsetminus U_{\alpha,j}^R$, and define $\eta_\alpha:X_\alpha \to H_\alpha=\oplus_{j\in J_\alpha}H_{\alpha,j}$, $\eta_\alpha(x)=\big(\eta_{\alpha,j}(x)\big)_{j\in J_\alpha}$, by 
	$$\eta_{\alpha,j}(x)=\varphi_{\alpha,j}(x)^{1/2} \xi_{\alpha,j}(x).$$ 
It now follows from~\cite[Proof of Theorem 3.2]{Dadarlat_Guentner2} that
	 \begin{enumerate}
	 	\item[(i$'$)] $\| \eta_\alpha(x) \|=1$, for all $x\in X_\alpha$ and $\alpha\in I$;
		\item[(ii$'$)] $\forall \alpha\in I$, $\forall x,y\in X_\alpha$, $d_\alpha(x,y)\leq R \; \Rightarrow \; \| \eta_\alpha(x)-\eta_\alpha(y) \|\leq \varepsilon$;
		\item[(iii$'$)] $\lim_{S \to \infty}\sup_{\alpha \in I}\sup \left\{ \big|\langle \eta_\alpha(x), \eta_\alpha(y) \rangle \big| : d_\alpha(x,y)\geq S, \; x,y\in X_\alpha \right\}=0$.
	 \end{enumerate}
Thus, by Proposition~\ref{prop:DG}, $\cX=\{X_\alpha \, | \, \alpha \in I\}$ is in $\fH$.
\end{proof}

\begin{remark}
	In~\cite[Theorem 4.6]{Guentner_Tessera_Yu2}, Guentner, Tessera and Yu proved that the collection of {\it exact}\footnote{{\it Exactness} of a metric space is a coarse invariant related to the notion of {\it Property~A}. Specifically, a metric space with Property~A is exact, and an exact metric space with bounded geometry has Property~A~\cite{Dadarlat_Guentner1}.} metric families, $\fE$, is closed under weak decomposition. Thus, since $\fB$ is contained in $\fE$, every metric family with weak finite decomposition complexity is also in~$\fE$. A straightforward generalization of~\cite[Proposition 2.10(c)]{Dadarlat_Guentner2} shows that an exact metric family is coarsely embeddable into Hilbert space. Therefore, we have the following sequence of inclusions of collections of metric families, each of which is stable under decomposition:
	\[ \wD \subset \fE \subset \fH. \] 
\end{remark}


\section{Weak Hyperbolic Dimension}

In this section we prove that a metric space with finite {\it hyperbolic dimension}, and more generally one with finite {\it weak hyperbolic dimension},
has weak finite decomposition complexity (Theorem~\ref{thm:weak_hyp_weak FDC}).
Buyalo and Schroeder introduced the hyperbolic dimension of a metric space (Definition \ref{def:hyperbolicdim}) to study the quasi-isometric embedding properties
of negatively curved spaces
(see  \cite{Buyalo_Schroeder} for an exposition).
The related notion of weak hyperbolic dimension was introduced by Cappadocia in
his Ph.D.~thesis, \cite{Cappadocia}.

\begin{definition}\label{def:doubling}
Let $N$ be a positive integer and $R>0$.
A subset $Y \subset X$ of a metric space $(X,d)$ is {\it $(N,R)$-large scale doubling}
if for every $x \in X$ and every $r\geq R$, the intersection of $Y$ with a ball in $X$ with radius $2r$ centered at $x$ can
be covered with $N$ balls of radius $r$ with centers in $X$.

A metric family $\cY$ of subsets of $X$ is {\it large scale doubling}\footnote{Some authors call such a collection of subsets {\it uniformly large scale doubling}.} if there exists $(N,R)$ such that
each $Y \in \cY$ is $(N,R)$-large scale doubling and
every finite union of elements of $\cY$ is $(N,R')$-large scale doubling,
where possibly $R' > R$ and $R'$ could depend on the particular finite union.
\end{definition}

Hyperbolic dimension is analogous to asymptotic dimension with the role of bounded metric families replaced
by large scale doubling metric families.

\begin{definition}\label{def:hyperbolicdim}
Let $n$ be a non-negative integer.
Let $\fL$ be the collection of large scale doubling metric families.
A metric space $(X,d)$  has {\it hyperbolic dimension}  at most $n$, denoted $\hyperdim(X)\leq n$, if $\{X\}$ is $n$-decomposable over $\fL$.
We say $\hyperdim(X) = n$ if $n$ is the smallest non-negative integer for which $\hyperdim(X)\leq n$.  
If no such integer exists then, by convention,  $\hyperdim(X) = \infty$.
\end{definition}

Since a bounded metric family is large scale doubling, we have that\linebreak 
 $\hyperdim(X) \leq \asdim(X)$.
If $X$ is a large scale doubling metric space (for example, $\R^n$ with the Euclidean metric), then  $\hyperdim(X) =0$.
Buyalo and Schroeder  showed $\hyperdim(\bbH^n) =n$, 
where  $\bbH^n$ is  $n$-dimensional hyperbolic space, $n \geq 2$.
Chris Cappadocia introduced the 
{\it weak hyperbolic dimension} of a metric space in his Ph.D.~thesis, \cite{Cappadocia}.
In Cappadocia's theory, large scale doubling metric families are replaced by {\it weakly large scale doubling}\footnote{Cappadocia uses the terminology {\it uniformly weakly large scale doubling}.}
metric families, dropping the condition on finite unions appearing in Definition~\ref{def:hyperbolicdim}. That is, a metric family $\cY$ of subsets of  a metric space $X$ is called {\it weakly large scale doubling} if there exists $(N,R)$ such that
each $Y \in \cY$ is $(N,R)$-large scale doubling.

\begin{definition}\label{def:weaklydoubling} (\cite{Cappadocia})
Let $w\fL$ be the collection of weakly large scale doubling metric families.
A metric space $(X,d)$  has {\it weak hyperbolic dimension}  at most $n$, denoted $\whyperdim(X)\leq n$, if $\{X\}$ is $n$-decomposable over $w\fL$.
We say $\whyperdim(X) = n$ if $n$ is the smallest non-negative integer for which $\whyperdim(X)\leq n$.  
If no such integer exists then, by convention,  $\whyperdim(X) = \infty$.
\end{definition}

Since $\fL \subset w\fL$, we have
$\whyperdim(X) \leq \hyperdim(X) \leq \asdim(X)$.

We say that a metric space is {\it $(N,R)$-large scale doubling} if it is $(N,R)$-large scale doubling as a subset of itself (see Definition \ref{def:doubling}).

\begin{lemma}
\label{lem:subsetdoubling}
Let $U \subset X$  be an $(N,R)$-large scale doubling subset of a metric space $(X, d_X)$.
Then $(U, d_U)$ is $(N^2,2R)$-large scale doubling where $d_U$ is the subspace metric induced by $d_X$.
\end{lemma}

\begin{proof}
Let $x \in U$ and $r \geq 2R$.
Since $U$ is an $(N,R)$-large scale doubling  subset of $X$, there are points $x_1, \ldots, x_{N^2} \in X$
such that $B_{2r}(x) \subset \bigcup^{N^2}_{i=1} B_{r/2}(x_i)$.
Let $J$ be the set of indices, $i$,  for which $B_{r/2}(x_i) \cap U$ is non-empty.
For each $i \in J$ choose $u_i \in B_{r/2}(x_i) \cap U$.
Since $B_{r/2}(x_i) \subset B_r(u_i)$ for $i \in J$,
we have  that 
$B_{2r}(x) \cap U \subset \bigcup_{i \in J} B_r(u_i) \cap U$.
\end{proof}

A subset $A \subset X$ of a metric space $(X, d_X)$ is said to be
{\it $L$-separated},  where $L > 0$, 
if $d_X(u,v) \geq L$ for all $u, v \in A$ with $u \neq v$.
We say that a metric space $(X, d_X)$ is {\it $N$-doubling}, where $N$ is a positive integer,
if it is $(N,R)$-large scale doubling for all $R>0$,  that is, doubling at all scales with doubling constant $N$.

%
%
%
%
%

We are now able to prove the key fact needed to establish Theorem~\ref{thm:weak_hyp_weak FDC}.

\begin{proposition}\label{prop:lsddecomposition}  
Let $\cX=\{ (X_\alpha, d_\alpha)  ~|~ \alpha \in I\}$ be a metric family such that there exists $(N,R)$ with the property that each $(X_\alpha, d_\alpha)$ is  $(N,R)$-large scale doubling. Then there exists a positive integer $M$, depending only on $N$, such that $\asdim(\cX) \leq M$.
\end{proposition}

\begin{proof}
Let $\lambda > 0$ be given.  Let $r = \max(\lambda, R)$.
For each $\alpha \in I$, choose a maximal $2r$-separated set $Z_\alpha \subset X_\alpha$.
Then $\cU_\alpha = \{ B_{4r}(x) ~|~ x \in Z_\alpha \}$ is a cover of $X_\alpha$.
Note that the Lebesgue number of $\cU_\alpha$ satisfies $L(\cU_\alpha) \geq \lambda$, for each $\alpha \in I$.

Let $\ell$ be a positive integer and assume that $y \in B_{4r}(x_1) \cap \cdots \cap B_{4r}(x_\ell)$,
where $x_1, \ldots, x_\ell \in Z_\alpha$ are distinct.
Note that $\{x_1, \ldots, x_\ell\} \subset B_{8r}(x_1)$.
By Lemma \ref{lem:subsetdoubling}, $B_{8r}(x_1) \cap Z_\alpha$ can be covered by $N^2$ balls of radius $4r$ with centers in $Z_\alpha$ and,
in turn, each of these balls can be covered by $N^2$ balls of radius $2r$ with centers in $Z_\alpha$.
For each $z \in Z_\alpha$, $B_{2r}(z) \cap  Z_\alpha = \{z\}$ because $Z_\alpha$ is $2r$-separated.
It follows that $B_{8r}(x_1) \cap  Z_\alpha$ contains at most $N^4$ points, and so $\ell \leq N^4$.
Hence, for each $\alpha \in I$, the multiplicity of the cover $\cU_\alpha$ is at most $N^4$. Since $\cup_{\alpha\in I}\cU_\alpha$ is a bounded metric family, Proposition~\ref{prop:equiv} implies that $\cX$ is $(N^4-1)$-decomposable over $\fB$ (the collection of bounded metric families). In other words, $\asdim(\cX)\leq N^4-1$.
\end{proof}

Combining Definition \ref{def:weaklydoubling} and  Proposition~\ref{prop:lsddecomposition} yields the following theorem.

\begin{theorem}\label{thm:weak_hyp_weak FDC}
A  metric space $X$ with finite weak hyperbolic dimension has weak finite decomposition complexity.
If $X$ has weak hyperbolic dimension at most 1, then $X$ has (strong) finite decomposition complexity.
\end{theorem}

\begin{proof}
If $\whyperdim(X)\leq n$, then, by Proposition~\ref{prop:lsddecomposition}, $X$ is $n$-decomposable over $\fA$ (the collection of metric families with finite asymptotic dimension). Therefore by \eqref{eq:asdim_has_FDC}, $X$ has weak FDC, and if $n\leq 1$, then $X$ has FDC.
\end{proof}

Since $\whyperdim(X) \leq \hyperdim(X)$, we also get the following corollary.

\begin{corollary}\label{cor:whyperdimdecompose}
A  metric space $X$ with finite hyperbolic dimension has weak finite decomposition complexity.
If $X$ has hyperbolic dimension at most 1, then $X$ has (strong) finite decomposition complexity. \qed
\end{corollary}



\section{Some open questions}

In this section we discuss some open problems involving decomposition complexity.

There are some interesting finitely generated groups for which the FDC (or weak FDC) condition is unknown.
 \smallskip
 
\begin{question}
\label{opengroupquestions}
Consider the following groups.
\begin{enumerate}
\item Grigorchuk's group  of intermediate growth,
\item Thompson's group $F = \langle A, B ~|~ [AB^{-1}, A^{-1}BA] = [AB^{-1}, A^{-2}BA^2] = 1 \rangle$,
\item $\operatorname{Out}(F_n)$, the outer automorphism group of a free group  $F_n$ of rank $n \geq 3$.
\end{enumerate}
For which of these groups, if any,  does the FDC (or weak FDC) condition hold?
\end{question}

Grigorchuk's group and Thompson's group $F$ are known to have infinite asymptotic dimension.
Grigorchuk's group is amenable and therefore has Yu's {\it  Property~A},  
a condition that implies coarse embeddability into Hilbert space
(see \cite[\S 4]{Guentner_Tessera_Yu2}
for a discussion of Property~A).
A group with weak FDC has Property~A, but the reverse implication is unknown.

\begin{question}
Does Property~A for a countable group imply weak FDC?
\end{question}

Osajda gave an example of a finitely generated group that is coarsely embeddable into Hilbert space yet does not have Property~A,  \cite{Osajda}.
Thus, Osajda's example is a group in the collection $\fH$ of metric families that are coarsely embeddable into Hilbert space, but not in the collection $\fE$ of exact metric families.

\begin{question}
Are there interesting collections of metric families, stable under (weak or strong) decomposition, that lie strictly in between $\fE$ and $\fH$?
\end{question}

As pointed out to us by the referee, while our proof of Theorem~\ref{thm:decomp-coarse-embedding} is very specific to Hilbert spaces, it is natural to try to generalize it to more general classes of Banach spaces. That is:

\begin{question}
Is there an interesting class of Banach spaces that is stable under weak decomposition? \end{question}

The mapping class group of a surface has finite asymptotic dimension, \cite{BBF},
and, by analogy,
one surmises that $\operatorname{Out}(F_n)$ may also have  finite asymptotic dimension and hence FDC.
Although a proof that the asymptotic dimension of $\operatorname{Out}(F_n)$ is finite has so far been elusive,
perhaps the less restrictive,
yet geometrically consequential (see the discussion in \S\ref{sec:DC}), weak FDC condition might be easier to demonstrate.

\begin{question}
\label{opengroupquestionstwo}
Which, if any, of the groups:  Grigorchuk's group,  Thompson's group $F$ and  $\operatorname{Out}(F_n)$,  $n \geq 3$, have finite weak hyperbolic dimension?
\end{question}
None of these groups are large scale doubling as metric spaces and so their weak hyperbolic dimension is
at least $1$.   Note that by Theorem \ref{thm:weak_hyp_weak FDC}, any group on this list that has finite weak hyperbolic dimension must have weak FDC.

\begin{question}
Does a space with finite weak hyperbolic dimension have FDC?
\end{question}
This question may be more tractable than the general question of whether weak FDC implies FDC
(see Question \ref{doesweakimplystrong}).



\def\cprime{$'$}


\begin{thebibliography}{Mos53}


\bibitem[BD04]{Bell_Dranish1}
G.~Bell and A.~Dranishnikov.
\newblock On asymptotic dimension of groups acting on trees.
\newblock {\em Geom. Dedicata}, {\bf 103} (2004), 89--101.


%

\bibitem[BD11]{Bell_Dranish2}
G.~Bell and A.~Dranishnikov.
\newblock Asymptotic dimension in Bedlewo.
\newblock {\em Topology Proc.}, {\bf 38} (2011), 209--236.

\bibitem[BBF10]{BBF}
M.~Bestvina, K.~Bromberg, and K.~Fujiwara.
\newblock Constructing group actions on quasi-trees and applications to mapping
  class groups,  arXiv:1006.1939v5,  2014.

\bibitem[BS07]{Buyalo_Schroeder}
S.~Buyalo and V.~Schroeder.
\newblock Elements of asymptotic geometry,
\newblock EMS Monographs in Mathematics, European Mathematical Society (EMS), Z\"urich, 2007. 

\bibitem[Cap14]{Cappadocia}
C.~Cappadocia.
\newblock {\em Large scale dimension theory of metric spaces}.
\newblock 2014.
\newblock Thesis (Ph.D.)~--~McMaster University.

\bibitem[DG03]{Dadarlat_Guentner1}
M.~Dadarlat and E.~Guentner.
\newblock Constructions preserving Hilbert space uniform embeddability of discrete groups.
\newblock {\em Trans. Amer. Math. Soc.}, {\bf 355} (2003), 3253--3275.

\bibitem[DG07]{Dadarlat_Guentner2}
M.~Dadarlat and E.~Guentner.
\newblock Uniform embeddability of relatively hyperbolic groups.
\newblock {\em J. Reine Angew. Math.}, {\bf 612} (2007), 1--15.


\bibitem[Gol13]{Goldfarb}
B.~Goldfarb.
\newblock Weak coherence of groups and finite decomposition complexity, arXiv:1307.5345, 2013.

\bibitem[Gra06]{Grave}
B.~Grave.
\newblock Asymptotic dimension of coarse spaces.
\newblock {\em New York J. Math.} {\bf12} (2006), 249--256 (electronic).


\bibitem[Gro93]{Gromov}
M.~Gromov.
\newblock Asymptotic invariants of infinite groups.
\newblock In {\em Geometric group theory, {V}ol.\ 2 ({S}ussex, 1991)}, volume
  182 of {\em London Math. Soc. Lecture Note Ser.}, pages 1--295. Cambridge
  Univ. Press, Cambridge, 1993.


\bibitem[GTY12]{Guentner_Tessera_Yu1}
E.~Guentner, R.~Tessera and G.~Yu. 
\newblock A notion of geometric complexity and its application to topological rigidity. 
\newblock {\em Invent. Math.} {\bf 7}(2) (2012), 315--357.


\bibitem[GTY13]{Guentner_Tessera_Yu2}
E.~Guentner, R.~Tessera and G.~Yu. 
\newblock Discrete groups with finite decomposition complexity. 
\newblock {\em Groups Geom. Dyn.} {\bf 7}(2) (2013), 377--402.

\bibitem[Kas15]{Kasprowski}
D.~Kasprowski.
\newblock On the $K$-theory of groups with finite decomposition complexity.  \newblock {\em Proc. Lond. Math. Soc.} {\bf 110}(3) (2015), 565--592.

\bibitem[Osa14]{Osajda}
D.~Osajda.
\newblock Small cancellation labellings of some infinite graphs and
  applications,  arXiv:1406.5015, 2014.

\bibitem[RTY14]{Ramras-Tessera-Yu}
D.~Ramras, R.~Tessera, and G.~Yu.
\newblock Finite decomposition complexity and the integral {N}ovikov conjecture
  for higher algebraic {$K$}-theory.
\newblock {\em J. Reine Angew. Math.}, 694:129--178, 2014.

\bibitem[Roe03]{Roe}
J.~Roe.
\newblock {\em Lectures on Coarse Geometry}, volume~31 of {\em University
  Lecture Series}.
\newblock American Mathematical Society, Providence, RI, 2003.

\bibitem[WC11]{WuChen}
Y.~Wu and X.~Chen.
\newblock On finite decomposition complexity of {T}hompson group.
\newblock {\em J. Funct. Anal.}, 261(4):981--998, 2011.

\bibitem[Yu98]{Yu}
G.~Yu. 
\newblock The Novikov conjecture for groups with finite asymptotic dimension. 
\newblock {\em Ann. Math.} {\bf 147}(2) (1998), 325--355.



\end{thebibliography}
\end{document}